\documentclass[12pt,english]{amsart}
\usepackage{ae,aecompl}
\usepackage[T1]{fontenc}
\usepackage[latin9]{inputenc}
\usepackage[a4paper]{geometry}
\usepackage{xcolor}
\usepackage{pdfcolmk}
\usepackage{verbatim}
\usepackage{mathtools}
\usepackage{amsthm}
\usepackage{amssymb}
\usepackage{stmaryrd}
\PassOptionsToPackage{normalem}{ulem}
\usepackage{ulem}

\makeatletter

\providecolor{lyxadded}{rgb}{0,0,1}
\providecolor{lyxdeleted}{rgb}{1,0,0}

\numberwithin{equation}{section}
\numberwithin{figure}{section}
\theoremstyle{plain}
\newtheorem{thm}{\protect\theoremname}[section]
  \theoremstyle{definition}
  \newtheorem{defn}[thm]{\protect\definitionname}
  \theoremstyle{remark}
  \newtheorem{rem}[thm]{\protect\remarkname}
  \theoremstyle{plain}
  \newtheorem{lem}[thm]{\protect\lemmaname}
  \theoremstyle{plain}
  \newtheorem{cor}[thm]{\protect\corollaryname}

\usepackage{babel}
\usepackage[all]{xy}
\usepackage[active]{srcltx}

\def\quot{/\!\!/}
\def\sym{\mathrm{Sym}}

\def\hom{\mathsf{Hom}}

\makeatother

\title[Hodge-Deligne polynomials of symmetric products of algebraic groups]{Hodge-Deligne polynomials of symmetric products of algebraic groups}

\author[J. Silva]{Jaime D. Silva}

\address{Departamento Matem\'{a}tica, FCUP,  Rua do Campo Alegre s/n, 4169-007 Porto, Portugal}

\email{jaime.silva@fc.up.pt}

\thanks{This work was supported by the project PTDC/MAT-GEO/2823/2014}



\makeatother

\usepackage{babel}
  \providecommand{\corollaryname}{Corollary}
  \providecommand{\definitionname}{Definition}
  \providecommand{\lemmaname}{Lemma}
  \providecommand{\remarkname}{Remark}
\providecommand{\theoremname}{Theorem}

\begin{document}
\begin{abstract}
Let $X$ be a complex quasi-projective algebraic variety. In this
paper we study the mixed Hodge structures of the symmetric products
$\sym^{n}X$ when the cohomology of $X$ is given by exterior products
of cohomology classes with odd degree. We obtain an expression for
the equivariant mixed Hodge polynomials $\mu_{X^{n}}^{S_{n}}\left(t,u,v\right)$,
codifying the permutation action of $S_{n}$ as well as its subgroups.
This allows us to deduce formulas for the mixed Hodge polynomials
of its symmetric products $\mu_{\sym^{n}X}\left(t,u,v\right)$. These
formulas are then applied to the case of linear algebraic groups.
\end{abstract}

\maketitle

\section{Introduction}

\thispagestyle{empty}

Given a topological space $X$, its $n$-fold symmetric product is
given by identifying in its $n$-fold cartesian product $X^{n}$ those
tuples that can be obtained from each other by permuting the entries.
More concretely, it is given by the finite quotient
\begin{eqnarray*}
\mbox{Sym}^{n}X & \coloneqq & X^{n}/S_{n},
\end{eqnarray*}
where $S_{n}$ is the symmetric group on $n$ letters that acts on
$X^{n}$ by permutation. When $X$ is a smooth complex algebraic curve,
its symmetric products are also smooth algebraic varieties \cite{ST}%
. For example, when $X$ is a compact Riemann surface $C$ of genus
$g$, $\mbox{Sym}^{g}C$ is birationally equivalent to the Jacobian
$J$ of $C$. Moreover, for $n>2g-2$ the symmetric product $\mbox{Sym}^{n}C$
is a projective fiber bundle over $J$ - see \cite{Mac2}. If one
assumes that $\dim_{\mathbb{C}}X>1$, then $\mbox{Sym}^{n}X$ are
no longer smooth but they are still quasi-projective algebraic varieties
- see \cite{Mum}. 

The cohomology of a complex quasi-projective algebraic variety $X$
is endowed with a natural mixed Hodge structure \cite{De2}. The mixed
Hodge numbers of $\sym^{n}X$ can be obtained from a formula of J.
Cheah \cite{Ch}, that generalized the work of I. G. Macdonald on
their Poincaré polynomials \cite{Mac}. 

In this paper, we provide an alternative approach to the problem of
determining the mixed Hodge polynomials of symmetric products of certain
classes of algebraic varieties whose cohomologies are exterior algebras
in a certain sense described below.

Mixed Hodge structures of $X$ define a triply-graded structure 
\[
H^{*}(X,\mathbb{C})=\bigoplus_{k,p,q}H^{k;p,q}\left(X,\mathbb{C}\right),
\]
satisfying the duality $H^{k;p,q}\left(X,\mathbb{C}\right)\cong\overline{H^{k;q,p}\left(X,\mathbb{C}\right)}$.
Fix $m\in\mathbb{N}$ and $\left(r_{1},\cdots,r_{m}\right)\in\mathbb{N}^{m}$.
We say that the cohomology of $X$ is an exterior algebra generated
in odd degree if there are classes $\omega_{j}^{i}\in H^{d_{i};p_{i},q_{i}}\left(X,\mathbb{C}\right)$,
for $i=1,\cdots,m$ and $j\in\{1,\cdots,r_{i}\}$, with $d_{i}\in2\mathbb{N}-1$,
$p_{i},q_{i}\in\mathbb{N}$, such that

\begin{eqnarray}
H^{*;*,*}\left(X,\mathbb{C}\right) & = & \bigwedge\left\langle \omega_{j}^{i}\right\rangle _{i,j=1}^{m,r_{i}}.\label{eq:exterior}
\end{eqnarray}
The class of varieties whose cohomology is of this form includes all
linear algebraic groups and abelian varieties. 

Our strategy follows by considering equivariant mixed Hodge polynomials
$\mu_{X}^{G}\left(t,u,v\right)$, codifying the triply-graded $G$-module
$\left[H^{*;*,*}\left(X,\mathbb{C}\right)\right]_{G}$ associated
to an algebraic action of a finite group $G$ on $X$. This gives
the main result of the paper. 
\begin{thm}
\label{thm:main} Let $X$ be a complex quasi-projective variety whose
cohomology is an exterior algebra generated in odd degree, of the
form in \eqref{eq:exterior}. Then the equivariant mixed Hodge polynomial
for the natural $S_{n}$ action on $X^{n}$ is given by
\begin{eqnarray*}
\mu_{X^{n}}^{S_{n}}\left(t,u,v\right) & = & \bigotimes_{i=1}^{m}\left[\sum_{k=0}^{n-1}\bigwedge^{k}\mbox{St}\left(\left(t^{d_{i}}u^{p_{i}}v^{q_{i}}\right)^{k}+\left(t^{d_{i}}u^{p_{i}}v^{q_{i}}\right)^{k+1}\right)\right]^{\varotimes r_{i}}
\end{eqnarray*}
where $\mbox{St}$ is the standard representation.
\end{thm}
With this Theorem, the deduction of the mixed Hodge polynomial of
$\mu_{\sym^{n}X}\left(t,u,v\right)$ follows from the representation
theory of the symmetric group $S_{n}$. 
\begin{thm}
\label{thm:det-formula}Let $X$ be a complex quasi-projective variety
whose cohomology is an exterior algebra generated in odd degree, as
in \eqref{eq:exterior}. Then, 
\begin{eqnarray*}
\mu_{\sym^{n}X}\left(t,u,v\right) & = & \frac{1}{n!}\sum_{\alpha\in S_{n}}\prod_{i=1}^{m}\det\left(I+t^{d_{i}}u^{p_{i}}v^{q_{i}}M_{\alpha}\right)^{r_{i}}
\end{eqnarray*}
for $M_{\alpha}$ the permutation matrix associated to $\alpha$. 
\end{thm}
To give an example, in the case of a complex torus $\mathbb{T}_{d}$
of dimension $d$, we get the formula

\begin{eqnarray}
\mu_{\sym^{n}\mathbb{T}_{d}}\left(t,u,v\right) & = & \frac{1}{n!}\sum_{\alpha\in S_{n}}\det\left(I+tuM_{\alpha}\right)^{d}\det\left(I+tvM_{\alpha}\right)^{d}\label{eq:ellmhp}
\end{eqnarray}
where $M_{\alpha}$ is the matrix of permutation of $\alpha\in S_{n}$. 

The above formula also gives the Poincaré polynomial by setting $u=v=1$.
In particular, for any algebraic variety as in Theorem \ref{thm:det-formula},
we get
\[
P_{\sym^{n}X}\left(t\right)=\frac{1}{n!}\sum_{\alpha\in S_{n}}\prod_{i=1}^{m}\det\left(I+t^{d_{i}}M_{\alpha}\right)^{r_{i}}.
\]

Formulas of this type are also important in the theory of character
varieties, which are algebraic varieties of the form $\mathcal{X}_{\Gamma}G:=\hom(\Gamma,G)\quot G$
for a finitely presented group $\Gamma$ and a complex reductive group
$G$. In the special case of $\Gamma=\mathbb{Z}^{r}$, the Poincaré
polynomials and, more generally, the mixed Hodge polynomials of $\mathcal{X}_{\Gamma}G$
were computed in \cite{St} and in \cite{FS}; for example in the
case $G=GL(n,\mathbb{C})$, they correspond to setting $d_{i}=p_{i}=q_{i}=1$. 

The formulas obtained here for mixed Hodge polynomials agree with
Cheah's formula (see Remark \ref{rem:frCh}), which provides the generating
series of the compactly supported mixed Hodge polynomials of $\sym^{n}X$,
\begin{eqnarray}
\sum_{n\geq0}\,\mu_{\sym^{n}X}^{c}(t,u,v)\,z^{n} & = & \prod_{p,q,k}\left(1-(-1)^{k}u^{p}v^{q}t{}^{k}z\right)^{(-1)^{k+1}h_{c}^{k,p,q}(X)}\label{eq:Cheah-2}
\end{eqnarray}
yielding a similar formula for the usual mixed Hodge numbers when
Poincaré duality applies. This formula can then be used to recover
$\mu_{\sym^{n}X}^{c}$, being given by the coefficient of $z^{n}$
in the right-hand side. 

We now outline the contents of the article. In Section 2, we prove
the main results (Theorem \ref{thm:equipol} and Theorem \ref{thm:MHP}).
We start by deducing the equivariant mixed Hodge polynomial by analyzing
the induced action of $S_{n}$ on $H^{*}\left(X^{n},\mathbb{C}\right)$.
Afterwards, by some simple considerations involving only the Schur
orthogonality relations, we deduce a general formula for $\mu_{\sym^{n}X}$.
In Section 3, we apply the results in Section 2 to several families
of examples. Most important among those are linear algebraic groups,
that motivated this study. In section 3.3, we also consider the case
of real topological Lie groups. Finally, in Section 4 we compare our
result with the above formula of J. Cheah, leading to some interesting
combinatorial identities (Theorem \ref{thm:combgl}) generalizing
those of \cite[Section 5.5]{FS}.

\section{Equivariant Polynomials of Permutation Actions}

Let $X$ be a complex quasi-projective variety and $S_{n}$ the symmetric
group on $n$ letters acting on $X^{n}$ by permutation. In this Section,
we explore the induced action $S_{n}\curvearrowright H^{k;p,q}\left(X^{n},\mathbb{C}\right)$
when the cohomology of $X$ is an exterior algebra generated in odd
degree. We will do so by determining the equivariant mixed Hodge polynomial
$\mu_{X^{n}}^{S_{n}}$. From it, we will be able to deduce a general
formula for $\mu_{\sym^{n}X}$ by a simple calculation on characters.
We start by giving an overview on mixed Hodge structures and their
relations to actions of finite groups.

\subsection{Equivariant mixed Hodge structures for finite group actions.}

The cohomology of a complex quasi-projective algebraic variety $X$
is endowed with a mixed Hodge structure \cite{De1,De2}. Briefly,
its cohomology $H^{*}\left(X,\mathbb{C}\right)$ admits two natural
filtrations: an increasing filtration, called the weight filtration
$W^{*}$, that can be defined over the rationals $H^{*}\left(G,\mathbb{Q}\right)$,
and a decreasing filtration, denoted $F_{*}$, generalizing the Hodge
filtration of smooth projective varieties. The name of these structures
is motivated from the fact that the Hodge filtration $F_{*}$ induces
a pure Hodge structure on the graded pieces of the weight filtration.
This leads to a bi-graduation of the cohomology ring, whose pieces
are called mixed Hodge components and are denoted by 
\begin{eqnarray*}
H^{k;p,q}\left(X,\mathbb{C}\right) & \coloneqq & Gr_{F}^{p}Gr_{p+q}^{W_{\mathbb{C}}}H^{k}\left(X,\mathbb{C}\right).
\end{eqnarray*}
Since the Hodge filtration induces a pure Hodge structure on $Gr_{p+q}^{W_{\mathbb{C}}}H^{k}\left(X,\mathbb{C}\right)$,
these pieces satisfy the duality $H^{k;p,q}\left(X,\mathbb{C}\right)\cong\overline{H^{k;q,p}\left(X,\mathbb{C}\right)}$.
Their dimensions are called \textit{mixed Hodge numbers} and are denoted
by 
\begin{eqnarray*}
h^{k;p,q}\left(X\right) & \coloneqq & \dim H^{k;p,q}\left(X,\mathbb{C}\right).
\end{eqnarray*}
These numbers are usually codified as a polynomial in three variables
\begin{eqnarray*}
\mu_{X}\left(t,u,v\right) & \coloneqq & \sum_{k,p,q}h^{k;p,q}\left(X\right)t^{k}u^{p}v^{q}
\end{eqnarray*}
known as \textit{mixed Hodge polynomial (MHP) or Hodge-Deligne polynomial
of $X$}. We recall that mixed Hodge structures also exist in the
compactly supported cohomology. Moreover, when $X$ is smooth (or
an orbifold) Poincaré duality is compatible with mixed Hodge structures,
so 
\begin{eqnarray*}
h^{k;p,q}\left(X\right) & = & h_{c}^{2d-k;k-p,k-q}\left(X\right)
\end{eqnarray*}
where $d=\dim_{\mathbb{C}}X$. In here and below the subscript (or
superscript, for the Hodge polynomials) $c$ means we are referring
to the compactly supported cohomology. Also important will be the
$E$ -polynomial, given by $\mu_{X}\left(-1,u,v\right)$ for the usual
cohomology and $\mu_{X}^{c}\left(-1,u,v\right)$ for the compactly
supported version 
\begin{eqnarray*}
E_{X}^{c}\left(u,v\right) & \coloneqq & \sum_{k,p,q}\left(-1\right)^{k}h_{c}^{k;p,q}\left(X\right)u^{p}v^{q}.
\end{eqnarray*}
Although this polynomial codifies less information than $\mu_{X}$,
it satisfies some very nice properties. Being an example of a motivic
measure on the category of complex quasi-projective varities, the
E-polynomial for the compactly supported cohomology is additive for
locally closed stratifications. In certain contexts, it also provides
a link with arithmetic geometry. An important result of this nature
is one by N. Katz in the Appendix of \cite{HRV}. Also, the E-polynomial
is also multiplicative for fibrations with trivial monodromy (see
\cite{LMN}). For some type of varieties, this polynomial is enough
to recover the mixed Hodge polynomial. This is the case of smooth
projective varieties, but also of separably pure varieties - a family
of varieties studied in \cite{DiLe}, and also essential in \cite{FS}. 

Now, let $X$ be a complex quasi-projective variety endowed with an
algebraic action of a finite group $G$. Since $G$ acts algebraically,
the induced action $G\curvearrowright H^{*}\left(X,\mathbb{C}\right)$
preserves mixed Hodge structures. Then the mixed Hodge components
are endowed with a $G$-module structure, that we denote by $\left[H^{k;p,q}\left(X,\mathbb{C}\right)\right]_{G}$.
This way, we can think of $H^{*;*,*}\left(X,\mathbb{C}\right)$ as
a triply graded $G$-module. Following a standard procedure, we codify
this action in a polynomial.
\begin{defn}
Let $X$ be a complex quasi-projective $G$-variety, for $G$ a finite
group. We define the \textit{equivariant mixed Hodge polynomial of
$G\curvearrowright X$} as 
\begin{eqnarray*}
\mu_{X}^{G}\left(t,u,v\right) & = & \sum_{k,p,q}\left[H^{k;p,q}\left(X,\mathbb{C}\right)\right]_{G}t^{k}u^{p}v^{q}.
\end{eqnarray*}

\end{defn}
The polynomials codifying actions of finite groups on the mixed Hodge
components were firstly introduced in \cite{DK}. The equivariant
mixed Hodge polynomial encodes all numerical information related to
this action. For instance, one can recover $\mu_{X}$ by taking 
\begin{eqnarray*}
\mu_{X}\left(t,u,v\right) & = & \dim_{\mathbb{C}}\mu_{X}^{G}\left(t,u,v\right).
\end{eqnarray*}
One can also recover the mixed Hodge polynomial of the quotient by
applying the isomorphism 
\begin{eqnarray}
H^{*;*,*}\left(X/G,\mathbb{C}\right) & \cong & H^{*;*,*}\left(X,\mathbb{C}\right)^{G}\label{eq:Groequ}
\end{eqnarray}
\cite{Gro}, that tells us $\mu_{X/G}$ equals the coefficient of
the trivial representation when $\mu_{X}^{G}$ is written in a irreducible
basis. By restricting the action to a subgroup $H\hookrightarrow G$,
we get 
\begin{eqnarray*}
\mu_{X}^{H}\left(t,u,v\right) & = & \mu_{X}^{G}\left(t,u,v\right)|_{H}\\
 & = & \sum_{k,p,q}\left[H^{k;p,q}\left(X,\mathbb{C}\right)\right]_{G}|_{H}t^{k}u^{p}v^{q}.
\end{eqnarray*}
Then we can also get information on the quotient by any subgroup.
For a more detailed account on mixed Hodge structures, see \cite{PS}.

\subsection{Permutation Actions}

Let $X$ be a complex quasi-projective variety and consider the permutation
action of $G=S_{n}$ on $X^{n}$. We are interested in the cases where
the cohomology of $X$ is an exterior algebra, as follows. 
\begin{defn}
\label{def:extalg}Fix $m\in\mathbb{N}$ and $\left(r_{1},\cdots,r_{m}\right)\in\mathbb{N}^{m}$.
We say that the cohomology of $X$ is an \textit{exterior algebra
generated in odd degree} if there are classes $\omega_{j}^{i}\in H^{d_{i};p_{i},q_{i}}\left(X,\mathbb{C}\right)$,
for $i=1,\cdots,m$ and $j\in\{1,\cdots,r_{i}\}$, with $d_{i}\in2\mathbb{N}-1$,
$p_{i},q_{i}\in\mathbb{N}_{0}$, such that
\begin{eqnarray*}
H^{*;*,*}\left(X,\mathbb{C}\right) & = & \bigwedge\left\langle \omega_{j}^{i}\right\rangle _{i,j=1}^{m,r_{i}}.
\end{eqnarray*}

\end{defn}
We will start by obtaining an expression for the equivariant mixed
Hodge polynomial $\mu_{X^{n}}^{S_{n}}$.
\begin{thm}
\label{thm:equipol}Let $X$ be a complex quasi-projective variety
whose cohomology is generared in odd degree. Assuming the conventions
in Definition \ref{def:extalg}, the equivariant mixed Hodge polynomial
for the natural $S_{n}$ action on $X^{n}$ is given by
\begin{eqnarray*}
\mu_{X^{n}}^{S_{n}}\left(t,u,v\right) & = & \bigotimes_{i=1}^{m}\left[\sum_{k=0}^{n-1}\bigwedge^{k}\mbox{St}\left(t^{d_{i}}u^{p_{i}}v^{q_{i}}\right)^{k}\left(1+t^{d_{i}}u^{p_{i}}v^{q_{i}}\right)\right]^{\varotimes r_{i}}
\end{eqnarray*}
where $\mbox{St}$ is the standard representation.\end{thm}
\begin{proof}
We will start by a suitable description of the cohomology of $X^{n}$.
By the Künneth isomorphism
\begin{eqnarray*}
H^{*;*,*}\left(X^{n},\mathbb{C}\right) & \cong & \left[\bigwedge\left\langle \omega_{j}^{i}\right\rangle _{i,j=1}^{m,r_{i}}\right]^{\varotimes n}\\
 & \cong & \left[\bigotimes_{i=1}^{m}\bigwedge\left\langle \omega_{1}^{i},\cdots,\omega_{r_{i}}^{i}\right\rangle _{\mathbb{C}}\right]^{\varotimes n}\\
 & \cong & \bigotimes_{j=1}^{n}\bigotimes_{i=1}^{m}\bigwedge\left\langle \omega_{1}^{i,j},\cdots,\omega_{r_{i}}^{i,j}\right\rangle _{\mathbb{C}}\\
 & \cong & \bigotimes_{i=1}^{m}\bigotimes_{j=1}^{n}\bigwedge\left\langle \omega_{1}^{i,j},\cdots,\omega_{r_{i}}^{i,j}\right\rangle _{\mathbb{C}}
\end{eqnarray*}
where $\omega_{k}^{i,j}$ is the image of $\omega_{k}^{i}$ in the
$j$-th component of $X^{n}$. Moreover, $S_{n}$ acts on $\bigotimes_{j=1}^{n}\bigwedge\left\langle \omega_{1}^{i,j},\cdots,\omega_{r_{i}}^{i,j}\right\rangle _{\mathbb{C}}\cong\bigwedge\bigotimes_{j=1}^{n}\left\langle \omega_{1}^{i,j},\cdots,\omega_{r_{i}}^{i,j}\right\rangle _{\mathbb{C}}$
by permutation on $j$, so 
\begin{eqnarray*}
\mu_{X^{n}}^{S_{n}}\left(t,u,v\right) & = & \bigotimes_{i=1}^{m}\left[\sum_{k=0}^{n-1}\bigwedge^{k}\rho_{S_{n}}\left(t^{d_{i}}u^{p_{i}}v^{q_{i}}\right)^{k}\right]^{\varotimes r_{i}}
\end{eqnarray*}
where $\rho_{S_{n}}=\mbox{T}+\mbox{St}$, for $\mbox{T}$ the trivial
and $\mbox{St}$ the standard representations. Then 
\begin{eqnarray*}
\sum_{k=0}^{n-1}\bigwedge^{k}\rho_{S_{n}}t^{d_{i}}u^{p_{i}}v^{q_{i}} & = & \sum_{k=0}^{n-1}\bigwedge^{k}\left[\mbox{St}+\mbox{T}\right]\left(t^{d_{i}}u^{p_{i}}v^{q_{i}}\right)^{k}\\
 & = & \sum_{k=0}^{n-1}\left[\bigwedge^{k}\mbox{St}+\bigwedge^{k-1}\mbox{St}\right]\left(t^{d_{i}}u^{p_{i}}v^{q_{i}}\right)^{k}\\
 & = & \sum_{k=0}^{n-1}\bigwedge^{k}\mbox{St}\left(\left(t^{d_{i}}u^{p_{i}}v^{q_{i}}\right)^{k}+\left(t^{d_{i}}u^{p_{i}}v^{q_{i}}\right)^{k+1}\right)
\end{eqnarray*}
and so the result follows.\end{proof}
\begin{rem}
The fact that cohomological forms of even degree $\omega$ satisfy
$\omega\wedge\alpha=\alpha\wedge\omega$ for any class $\alpha$,
explains why we have to require that the generators of $H^{*}\left(X,\mathbb{C}\right)$
have odd degree.
\end{rem}
In order to recover $\mu_{\sym^{n}X}$ from $\mu_{X^{n}}^{S_{n}}$,
we will need some simple facts from representation theory of finite
groups. Since this Theorem expresses the equivariant polynomial $\mu_{X^{n}}^{S_{n}}$
as a tensor product, we require a suitable way to obtain the coefficient
of the trivial representation from the product of two representations.
This is provided by the following Lemma. 
\begin{lem}
\label{lem:diaglem}Let $G$ be a finite group and $a=\left[\sum_{l=1}^{k}a_{l}T_{l}\right]$,
$b=\left[\sum_{l=1}^{k}b_{l}T_{l}\right]$ the decomposition into
irreducibles of two $G$-representations. Then the coefficient of
the trivial representation of $a\varotimes b$ is given by
\[
\frac{1}{\left|G\right|}\sum_{l=1}^{k}\left|\left[c_{l}\right]\right|\left(c_{1,l}a_{1}+\cdots+c_{k,l}a_{k}\right)\left(c_{1,l}b_{1}+\cdots+c_{k,l}b_{k}\right)
\]
where $C_{G}=\left(c_{i,j}\right)_{i,j}$ is the character table of
$G$ and $\left|\left[c_{l}\right]\right|$ is the order of the conjugation
class corresponding to the $l$-th column. We conclude, in particular,
that the coefficient of the trivial representation of $a^{\varotimes n}$,
for $n\geq2$, is given by
\[
\frac{1}{\left|G\right|}\sum_{l=1}^{k}\left|\left[c_{l}\right]\right|\left(c_{1,l}a_{1}+\cdots+c_{k,l}a_{k}\right)^{n}.
\]
\end{lem}
\begin{proof}
Denote by $\chi_{a\varotimes b}$ $\chi_{a^{\varotimes n}}$and the
characters of $a\varotimes b$ and $a^{\varotimes n}$. Then 
\begin{eqnarray*}
\chi_{a\varotimes b}\left(\left[c_{l}\right]\right) & = & \left(c_{1,l}a_{1}+\cdots+c_{k,l}a_{k}\right)\left(c_{1,l}b_{1}+\cdots+c_{k,l}b_{k}\right)\\
\chi_{a^{\varotimes n}}\left(\left[c_{l}\right]\right) & = & \left(c_{1,l}a_{1}+\cdots+c_{k,l}a_{k}\right)^{n}
\end{eqnarray*}
by the linear and multiplicative property of characters. So we can
recover their coefficient of the trivial representation by taking
the first entry of the vectors $v=C_{G}^{-1}\left(\chi_{a\varotimes b}\left(\left[c_{l}\right]\right)\right)_{l=1}^{k}$
and $u=C_{G}^{-1}\left(\chi_{a^{n}}\left(\left[c_{l}\right]\right)\right)_{l=1}^{k}$.
Explicitly, for $v$ we have 
\begin{eqnarray*}
\left(v\right)_{1} & = & \left(\frac{\left|\left[c_{1}\right]\right|}{\left|G\right|},\cdots,\frac{\left|\left[c_{k}\right]\right|}{\left|G\right|}\right)\times v^{t}\\
 & = & \frac{1}{\left|G\right|}\sum_{l=1}^{k}\left|\left[c_{l}\right]\right|\left(c_{1,l}a_{1}+\cdots+c_{k,l}a_{k}\right)\left(c_{1,l}b_{1}+\cdots+c_{k,l}b_{k}\right)
\end{eqnarray*}
where the first equality follows from Schur orthogonality relations
for columns. By the same reasoning, we also deduce the coefficient
of $a^{\varotimes n}$.
\end{proof}
We now proceed to the deduction of $\mu_{\sym^{n}X}$. From the isomorphism
\eqref{eq:Groequ}, we know the mixed Hodge polynomial equals the
coefficient of the trivial representation when one writes $\mu_{X^{n}}^{S_{n}}$
in a basis of irreducible representations of $S_{n}$. Since the equivariant
polynomial in Theorem \ref{thm:equipol} is expressed in terms of
exterior products of the standard representation, that are irreducible,
the previous Lemma will suffice. 
\begin{thm}
\label{thm:MHP}Let $X$ be a complex quasi-projective variety whose
cohomology is an exterior algebra generated in odd degree, as in Defnition
\ref{def:extalg}. Then, 
\begin{eqnarray*}
\mu_{\sym^{n}X}\left(t,u,v\right) & = & \frac{1}{n!}\sum_{\alpha\in S_{n}}\prod_{i=1}^{m}\det\left(I+t^{d_{i}}u^{p_{i}}v^{q_{i}}M_{\alpha}\right)^{r_{i}}
\end{eqnarray*}
for $M_{\alpha}$ the permutation matrix associated to $\alpha$. \end{thm}
\begin{proof}
In this proof, we will follow the notations of the previous Lemma.
We start by applying the second equality in Lemma \ref{lem:diaglem}
to the equivariant polynomial in Theorem \ref{thm:equipol}. We obtain
\begin{eqnarray*}
\mu_{\sym^{n}X}\left(t,x\right) & = & \frac{1}{n!}\sum_{l=1}^{p\left(n\right)}\left|\left[c_{l}\right]\right|\prod_{i=1}^{m}p_{l}\left(t^{d_{i}},u^{p_{i}},v^{q_{i}}\right)^{r_{i}}
\end{eqnarray*}
where $p_{l}\left(t,u,v\right)=\left(c_{1,l}a_{1}\left(t,u,v\right)+\cdots+c_{k,l}a_{k}\left(t,u,v\right)\right)$.
In here, the polynomials $a_{i}\left(t,u,v\right)$ are the coefficients
of the irreducible representations in 
\[
\sum_{k=0}^{n-1}\left(\left(tuv\right)^{k}+\left(tuv\right)^{k+1}\right)\bigwedge^{k}\mbox{St}.
\]
So $a_{i}=0$ unless $i$ corresponds to a exterior product of the
standard representation. Let $\rho_{S_{n}}=\mbox{St}+\mbox{T}$. Since
\begin{eqnarray*}
\det\left(I+x\rho_{S_{n}}\left(c_{l}\right)\right) & = & \sum_{k=0}^{\infty}\chi_{\bigwedge^{k}\rho_{S_{n}}}\left(\left[c_{l}\right]\right)x^{k}\\
 & = & \sum_{k=0}^{\infty}\left[\chi_{\bigwedge^{k}\mbox{St}}+\chi_{\bigwedge^{k-1}\mbox{St}}\right]\left(\left[c_{l}\right]\right)x^{k}\\
 & = & \sum_{k=0}^{n-1}\left[\chi_{\bigwedge^{k}\mbox{St}}\left(\left[c_{l}\right]\right)\left(x^{k}+x^{k+1}\right)\right]
\end{eqnarray*}
the result follows.
\end{proof}

\section{Applications to Algebraic Groups}

In this Section, we explore some applications of the main results
in the previous one (Theorem \ref{thm:equipol} and Theorem \ref{thm:MHP}).
We will focus on the case of algebraic groups and complex tori. This
article focus on the complex algebraic case, where symmetric products
have interesting properties, such as their connections to motivic
theory, local Zeta functions or to the Hilbert scheme of points. As
will be observed, the techniques employed in here also adapt to the
case of Lie groups, and we will study those in the final subsection.

\subsection{Complex Tori}

A complex torus $\mathbb{T}_{d}$ of dimension $d$ is a complex manifold
of the form $\mathbb{C}^{d}/L$, where $L$ is a lattice (subgroup
of $(\mathbb{C}^{d},+)$ isomorphic to $\mathbb{Z}^{2d}$). Being
compact and Kähler manifolds, the cohomology of $\mathbb{T}_{d}$
is endowed with a pure Hodge structure.
\begin{rem}
Whenever $L$ satisfies the Riemann bilinear relations, we get an
embedding of $\mathbb{C}^{d}/L$ into some projective space. Then
in these cases $\mathbb{T}_{d}$ forms an abelian variety of dimension
$d$.\end{rem}
\begin{thm}
\label{thm:AbVars}Let $X=\mathbb{T}_{d}$ be a complex torus of dimension
$d$. The equivariant mixed Hodge polynomial related to the permutation
action $S_{n}\curvearrowright X^{n}$ is given by 
\begin{eqnarray*}
\mu_{\mathbb{T}_{d}^{n}}^{S_{n}}\left(t,u,v\right) & = & \left[\sum_{k=0}^{n-1}\bigwedge^{k}\mbox{St}\left(tu\right)^{k}\left(1+tu\right)\right]^{\varotimes d}\varotimes\left[\sum_{k=0}^{n-1}\bigwedge^{k}\mbox{St}\left(tv\right)^{k}\left(1+tv\right)\right]^{\varotimes d}.
\end{eqnarray*}
\end{thm}
\begin{proof}
The singular cohomology of a complex torus $X$ is characterized by
\begin{eqnarray*}
H^{*}\left(X,\mathbb{Q}\right) & \cong & \bigwedge H^{1}\left(X,\mathbb{Q}\right).
\end{eqnarray*}
The pure Hodge structure in $H^{*}\left(X,\mathbb{Q}\right)$ is completely
determined by the previous equality 
\begin{eqnarray*}
H^{p,q}\left(X\right) & \cong & \left(\bigwedge^{p}H^{1,0}\left(X\right)\right)\varotimes\left(\bigwedge^{q}H^{0,1}\left(X\right)\right).
\end{eqnarray*}
Moreover, since $H_{1}\left(\mathbb{T}_{d},\mathbb{Z}\right)$ can
be identified with the lattice of $\mathbb{T}_{d}$, we have $H^{1,0}\left(\mathbb{T}_{d}\right)\cong\left\langle \omega_{1},\cdots,\omega_{d}\right\rangle _{\mathbb{C}}\cong\overline{H^{0,1}\left(\mathbb{T}_{d}\right)}$
for certain cohomological classes $\omega_{i}$ of degree 1. Then
we can apply the conditions in Theorem \ref{thm:equipol}, that gives
us 
\begin{eqnarray*}
\mu_{\mathbb{T}_{d}^{n}}^{S_{n}}\left(t,u,v\right) & = & \left[\sum_{k=0}^{n-1}\bigwedge^{k}\mbox{St}\left(\left(tu\right)^{k}+\left(tu\right)^{k+1}\right)\right]^{\varotimes d}\varotimes\left[\sum_{k=0}^{n-1}\bigwedge^{k}\mbox{St}\left(\left(tv\right)^{k}+\left(tv\right)^{k+1}\right)\right]^{\varotimes d},
\end{eqnarray*}
as wanted.\end{proof}
\begin{cor}
For $X=\mathbb{T}_{d}$ a complex torus of dimension $d$ we have:
\begin{eqnarray*}
\mu_{\sym^{n}\mathbb{T}_{d}}\left(t,u,v\right) & = & \frac{1}{n!}\sum_{\alpha\in S_{n}}\det\left(I+tuM_{\alpha}\right)^{d}\det\left(I+tvM_{\alpha}\right)^{d}
\end{eqnarray*}
where $M_{\alpha}$ is the matrix of permutation of $\alpha$.\end{cor}
\begin{proof}
This follows immediately from the Theorem and from Theorem \ref{thm:MHP}.
\end{proof}

\subsection{Linear algebraic groups}

We now proceed to the case of linear algebraic groups. These are subgroups
of some $GL\left(n,\mathbb{C}\right)$ given by polynomial equations.
If $G$ is a Lie group, besides the usual cup product the cohomology
ring is also endowed with a natural co-product. This is given by the
pullback of the product map composed with the Künneth isomorphism,
\[
\Delta:H^{*}\left(G,\mathbb{C}\right)\to H^{*}\left(G\times G,\mathbb{C}\right)\cong H^{*}\left(G,\mathbb{C}\right)\varotimes H^{*}\left(G,\mathbb{C}\right).
\]
We also have an antipode map, given by the pullback of the inverse
map $g\mapsto g^{-1}$. It is a well known fact that the cohomology
with these three operations forms a finitely generated Hopf algebra\footnote{Actually the cohomology - or homology, since Hopf algebras are self-dual
- of Lie groups were the inspiration behind the definition of Hopf
algebra.}. From Hopf's Theorem (see \cite{Hop}), we know there exists a set
$\left\{ \omega_{1},\cdots,\omega_{k}\right\} $ of classes of homogeneous
elements of odd degree that generate $H^{*}\left(G,\mathbb{C}\right)$
for the cup product. Since these elements have odd degree, one has
\begin{eqnarray}
H^{*}\left(G,\mathbb{C}\right) & \cong & \bigwedge\left\langle \omega_{1},\cdots,\omega_{k}\right\rangle _{\mathbb{C}}.\label{eq:primten}
\end{eqnarray}
These elements $\omega_{i}$ are the primitive elements of $H^{*}\left(G,\mathbb{C}\right)$
- those satisfying $\Delta\left(x\right)=x\varotimes1+1\varotimes x$.
To obtain the characterization of the mixed Hodge structure on their
symmetric products, we will use a result by Deligne on the seminal
paper on mixed Hodge structures of singular varieties \cite{De2}. 
\begin{thm}
\label{thm:LAGs}Let $G$ be a complex linear algebraic group. Then
there exists $m\in\mathbb{N}$ and $r_{1},\cdots,r_{m}\in\mathbb{N}_{0}^{m}$
s.t.
\begin{eqnarray*}
\mu_{G^{n}}^{S_{n}}\left(t,u,v\right) & = & \bigotimes_{i=1}^{m}\left[\sum_{k=0}^{n-1}\bigwedge^{k}\mbox{St}\left(\left(t^{2i-1}u^{i}v^{i}\right)^{k}+\left(t^{2i-1}u^{i}v^{i}\right)^{k+1}\right)\right]^{\varotimes r_{i}}
\end{eqnarray*}
where $\mbox{St}$ is the standard representation. Consequently,
\begin{eqnarray*}
\mu_{\sym^{n}G}\left(t,u,v\right) & = & \frac{1}{n!}\sum_{\alpha\in S_{n}}\prod_{i=1}^{m}\det\left(I+t^{2i-1}\left(uv\right)^{i}M_{\alpha}\right)^{r_{i}}.
\end{eqnarray*}
\end{thm}
\begin{proof}
Being a quasi-projective algebraic variety, the cohomology of $G$
is endowed with a mixed Hodge structure, that was characterized in
\cite[Theorem 9.1.5]{De2}. This result states that the vector space
$P^{*}=\left\langle \omega_{1},\cdots,\omega_{k}\right\rangle _{\mathbb{C}}$
generated by primitive elements is a sub-mixed Hodge structure of
$H^{*}\left(X,\mathbb{C}\right)$. Then the isomorphism \ref{eq:primten}
preserves mixed Hodge structures. Moreover, Deligne also showed that
for every $i=0,\cdots,\dim_{\mathbb{C}}G$, $P^{2i}\cong0$ and the
mixed Hodge structure on each $P^{2i-1}$ is pure of weights $\left(i,i\right)$.
If we divide $P^{*}$ in its graded components, we get 
\begin{eqnarray*}
H^{*;*,*}\left(G,\mathbb{C}\right) & \cong & \bigotimes_{i=1}^{m}\bigwedge\left\langle \omega_{i,1},\cdots,\omega_{i,r_{i}}\right\rangle _{\mathbb{C}}
\end{eqnarray*}
where on the right hand side we use the only multi-grading compatible
with the tensor and exterior products obtained by letting $\omega_{i,k}\in H^{2i-1,i,i}\left(G,\mathbb{C}\right)$,
$\forall k=1,\cdots,r_{i}$. From this isomorphism, we can read the
mixed Hodge polynomial of $G$, 
\begin{eqnarray*}
\mu_{G}\left(t,u,v\right) & = & \prod_{i=1}^{m}\left(1+t^{2i-1}\left(uv\right)^{i}\right)^{r_{i}}.
\end{eqnarray*}
If we apply Theorem \ref{thm:equipol}, we get an expression for the
equivariant polynomial of the permutation action $S_{n}\curvearrowright G^{n}$:
\begin{eqnarray*}
\mu_{G^{n}}^{S_{n}}\left(t,u,v\right) & = & \bigotimes_{i=1}^{m}\left[\sum_{k=0}^{n-1}\bigwedge^{k}\mbox{St}\left(\left(t^{2i-1}u^{i}v^{i}\right)^{k}+\left(t^{2i-1}u^{i}v^{i}\right)^{k+1}\right)\right]^{\varotimes r_{i}}.
\end{eqnarray*}
By Theorem \ref{thm:MHP}, we can also obtain a formula for the mixed
Hodge polynomial of the quotient
\begin{eqnarray}
\mu_{\sym^{n}G}\left(t,u,v\right) & = & \frac{1}{n!}\sum_{\alpha\in S_{n}}\prod_{i=1}^{m}\det\left(I+t^{2i-1}\left(uv\right)^{i}M_{\alpha}\right)^{r_{i}}\label{eq:MHPLAGs}
\end{eqnarray}
where $M_{\alpha}$ is the matrix of permutation of $\alpha$. \end{proof}
\begin{rem}
We have the following remarks concerning this example:
\begin{enumerate}
\item If we consider the particular case $G=\left(\mathbb{C}^{*}\right)^{r}$,
we recover the result in \cite[Proposition 5.8]{FS}. In here, we
obtained the mixed Hodge polynomial of the free abelian character
variety $\mathcal{M}_{r}GL\left(n,\mathbb{C}\right)\cong\sym^{n}\left(\mathbb{C}^{*}\right)^{r}$
(\cite{FL2,Sik}). So Theorem \ref{thm:LAGs} provide an alternative
proof of this result;
\item The cohomology ring of a connected linear algebraic group $G$ coincides
with a product of $\mathbb{C}^{m}\backslash\left\{ 0\right\} $. For
this reason, the formula above also calculates the mixed Hodge polynomial
of the symmetric product of arbitrary products of punctured complex
vector spaces. 
\end{enumerate}
\end{rem}

\subsection{The topological case}

In the final part of this Section, we deal with the Poincaré polynomial
related to the singular cohomology of symmetric products of real topological
Lie groups. The cohomology of Lie groups is an old theme of research,
where much is known, but we could not found these formulas in the
literature. 

As mentioned in the Introduction, the results in the Theorems \ref{thm:AbVars}
and \ref{thm:LAGs} allows us to deduce the Poincaré polynomial of
these spaces by setting 
\begin{eqnarray*}
P_{X}\left(t\right) & = & \mu_{X}\left(t,1,1\right).
\end{eqnarray*}
On the other hand, the techniques employed in the proofs of the Theorems
\ref{thm:equipol} and \ref{thm:MHP} are combinatorial, and adapt
well to the topological case. In this setting, we are concerned with
those topological spaces $X$ whose singular cohomology satisfies
a similar property to that of Definition \ref{def:extalg}: to say
that the cohomology of $X$ is an exterior algebra generated in odd
degree we require the existence of a finite set of cohomology classes
of odd degree whose exterior products generate the cohomology ring,
as in that definition. We now formalize this in a way that allows
us to keep track of the grading.
\begin{defn}
Fix $m\in\mathbb{N}$ and $\left(r_{1},\cdots,r_{m}\right)\in\mathbb{N}^{m}$.
We say that the singular cohomology of a topological space $X$ is
an \textit{exterior algebra generated in odd degree} if there are
classes $\omega_{j}^{i}\in H^{d_{i}}\left(X\right)$, for $i=1,\cdots,m$
and $j\in\{1,\cdots,r_{i}\}$, with $d_{i}\in2\mathbb{N}-1$, such
that
\begin{eqnarray*}
H^{*}\left(X\right) & = & \bigwedge\left\langle \omega_{j}^{i}\right\rangle _{i,j=1}^{m,r_{i}}.
\end{eqnarray*}

\end{defn}
In this definition $r_{i}$ counts the number of generators of degree
$d_{i}$: similarly, in Definition \ref{def:extalg} $r_{i}$ counted
the number of generators of odd degree $d_{i}$ and weights $\left(p_{i},q_{i}\right)$.
Lie groups form an important class of real topological spaces whose
cohomology is as in this definition - see the first remarks in 3.2. 
\begin{thm}
Let $G$ be a topological Lie group and admit the permutation action
of $S_{n}$ in its cartesian product $G^{n}$. Then there exists $m\in\mathbb{N}$
and $r_{1},\cdots,r_{m}\in\mathbb{N}_{0}^{m}$ such that
\begin{eqnarray*}
P_{\sym^{n}G}\left(t\right) & = & \frac{1}{n!}\sum_{\alpha\in S_{n}}\prod_{i=1}^{m}\det\left(I+t^{2i-1}M_{\alpha}\right)^{r_{i}}
\end{eqnarray*}
where $M_{\alpha}$ is the permutation matrix associated to $\alpha$. \end{thm}
\begin{proof}
The isomorphism \eqref{eq:Groequ}, identifying the cohomology of
a finite quotient with its invariant part, applies whenever we are
working in a sheaf cohomology theory with values in a field of characteristic
0 (actually, any field whose characteristic does not divide the order
of the finite group in question). Since the singular cohomology of
a topological Lie group $G$ can be identified with their sheaf cohmology
with $\mathbb{Q}$-coefficients, to prove this result it suffices
to mirror the proofs of Theorems \ref{thm:AbVars}, \ref{thm:LAGs}
and \ref{thm:LAGs}, replacing the equivariant mixed Hodge polynomial
by an equivariant Poincaré polynomial 
\begin{eqnarray*}
P_{G^{n}}^{S_{n}}\left(t\right) & := & \sum_{k}\left[H^{k}\left(G^{n}\right)\right]_{S_{n}}t^{k}
\end{eqnarray*}
codifying the induced action of $S_{n}$ on the cohomology groups
$H^{k}\left(G^{n}\right)$. \end{proof}
\begin{rem}
This formula agrees with the formula in \cite[Theorem 1.4]{St}, that
includes the case $G=\left(\mathbb{C}^{*}\right)^{r}$. Observe that
the Theorem above is valid for any topological space where the isomorphism
\eqref{eq:Groequ} applies. 
\end{rem}

\section{Combinatorial identities}

In this last Section we will use Theorem \ref{thm:equipol} and a
formula of J. Cheah (\cite{Ch}) to introduce some combinatorial identities,
in a similar way to \cite[Theorem 5.31]{FS}. Since the equalities
we will obtain only deal with the Betti numbers, we will focus on
the case of linear algebraic groups, covered here in Theorem \ref{thm:LAGs}.
For this, we will follow the same procedure as in \cite[Sections 5.5 and 5.6]{FS}.
As for this result, we could not find out whether these identities
were noticed before.

Let us start by reviewing the key ideas covered in \cite[Sections 5.5]{FS}.
As in here, fix $X$ as a quasi-projective algebraic variety with
given compactly supported Hodge numbers $h_{c}^{k;p,q}$. The mentioned
formula of J. Cheah is given by
\begin{equation}
\sum_{n\geq0}\,\mu_{\sym^{n}X}^{c}(t,u,v)\,z^{n}=\prod_{p,q,k}\left(1-(-1)^{k}u^{p}v^{q}t{}^{k}z\right)^{(-1)^{k+1}h_{c}^{k,p,q}(X)},\label{eq:Cheah}
\end{equation}
so it gives the generating function of the mixed Hodge polynomials
of all symmetric products $\sym^{n}X$. If one assumes that $X$ satisfies
a version of Poincaré duality compatible with mixed Hodge structures,
as it happens for smooth varieties or orbifolds, a first simple observation
is that this formula maintains unaltered when passing from $\mu^{c}$
to $\mu$ and from $h_{c}^{k,p,q}$ to $h^{k,p,q}$
\begin{equation}
\sum_{n\geq0}\,\mu_{\sym^{n}X}(t,u,v)\,z^{n}=\prod_{p,q,k}\left(1-(-1)^{k}u^{p}v^{q}t{}^{k}z\right)^{(-1)^{k+1}h^{k,p,q}(X)}\label{eq:Cheah-1}
\end{equation}
(see \cite[Proposition 5.22]{FS}). 
\begin{rem}
\label{rem:frCh} As remarked in the Introduction, the formulas obtained
in Theorem \ref{thm:MHP} can be obtained from Cheah's identity. Indeed,
following the notations in the proof of this Theorem, we have 
\begin{eqnarray*}
\prod_{i=1}^{m}\det\left(I+t^{d_{i}}u^{p_{i}}v^{q_{i}}\rho_{S_{n}}\left(\underline{n}\right)\right)^{r_{i}} & = & \prod_{i=1}^{m}p_{\underline{n}}\left(t^{d_{i}},u^{p_{i}},v^{q_{i}}\right)^{r_{i}}\\
 & = & \prod_{i=1}^{m}\prod_{j=1}^{n}\left(1-\left(-t^{d_{i}}u^{p_{i}}v^{q_{i}}\right)^{j}\right)^{a_{j}r_{i}}\\
 & = & \prod_{j=1}^{n}\left[\prod_{i=1}^{m}\left(1+\left(-\left(-t\right)^{j}\right)^{p_{i}}u^{jp_{i}}v^{jq_{i}}\right)^{r_{i}}\right]^{a_{j}}\\
 & = & \prod_{j=1}^{n}\mu_{X}\left(-\left(-t\right)^{j},u^{j},v^{j}\right)^{a_{j}}.
\end{eqnarray*}
On the other hand, the equality
\begin{eqnarray*}
\mu_{\sym^{n}X}\left(t,u,v\right) & = & \sum_{\underline{n}\in\mathcal{P}_{n}}\prod_{j=1}^{n}\frac{1}{a_{j}!\,j^{a_{j}}}\mu_{X}\left(-\left(-t\right)^{j},u^{j},v^{j}\right)^{a_{j}}
\end{eqnarray*}
can also be deduced by extracting the $z^{n}$ term in Cheah's formula
\eqref{eq:Cheah-1} (see also \cite{FS}). The approach we follow
in this article is independent of Cheah's formula, and is designed
to obtain more concrete and elegant formulas for the case of varieties
whose cohomology is an exterior algebra generated in odd degree. 
\end{rem}
Now fix $G$ as a connected complex linear algebraic group with mixed
Hodge polynomial $\mu_{G}\left(t,x\right)=\prod_{i=1}^{m}\left(1+t^{2i-1}x^{i}\right)^{r_{i}}$.
By comparing this formula with the equality \ref{eq:MHPLAGs}, one
obtains
\begin{equation}
\sum_{n\geq0}\frac{1}{n!}\sum_{\sigma\in S_{n}}\prod_{i=1}^{m}\det\left(I_{n}+t^{2i-1}x^{i}M_{\sigma}\right)^{r_{i}}z^{n}=\prod_{p,k}\left(1-(-1)^{k}x^{p}t{}^{k}z\right)^{(-1)^{k+1}h^{k,p,p}(G)}\label{eq:CheahFLS}
\end{equation}
The notation here is the same as the ones used on that example. Let
us focus for a while in the particular case $\left(\mathbb{C}^{*}\right)^{r}$,
handled in \cite{FS}. The related mixed Hodge polynomial (MHP) is
given by $\mu_{\left(\mathbb{C}^{*}\right)^{r}}\left(t,x\right)=\left(1+tx\right)^{r}$,
so $\left(\mathbb{C}^{*}\right)^{r}$ is a round variety. In particular,
its Betti numbers $b_{k}\left(\left(\mathbb{C}^{*}\right)^{r}\right)$
coincide with its only non-trivial mixed Hodge numbers $h^{k,k,k}\left(\left(\mathbb{C}^{*}\right)^{r}\right)$.
Moreover, given the form of the MHP of $\left(\mathbb{C}^{*}\right)^{r}$,
we have 
\begin{equation}
\begin{array}{ccccc}
b_{k}\left(\left(\mathbb{C}^{*}\right)^{r}\right) & = & h^{k,k,k}\left(\left(\mathbb{C}^{*}\right)^{r}\right) & = & \binom{r}{k}\end{array}.\label{eq:comb}
\end{equation}
Then replacing in equality \ref{eq:CheahFLS}, we get the combinatorial
identity of \cite[Theorem 5.31]{FS}. For the case of a general linear
group $G$, not only we do not have an equality between mixed Hodge
and Betti numbers, we also do not have a nice combinatorial interpretation
as that of equations \ref{eq:comb}. But for some groups, such as
$G=GL\left(n,\mathbb{C}\right)$, we do manage to interpret the Betti
numbers in a combinatorial fashion, justifying the next result.
\begin{lem}
\label{lem:Bet}Let $G$ be a connected complex linear algebraic group
whose mixed Hodge polynomial is $\mu_{G}\left(t,x\right)=\prod_{i=1}^{m}\left(1+t^{2i-1}x^{i}\right)^{r_{i}}$.
Denote by $M_{\sigma}$ the permutation matrix (in some basis) associated
to a permutation $\sigma\in S_{n}$. Then 
\begin{eqnarray*}
\sum_{n\geq0}\sum_{\sigma\in S_{n}}\frac{z^{n}}{n!}\prod_{i=1}^{m}\det\left(I_{n}+t^{2i-1}M_{\sigma}\right)^{r_{i}} & = & \prod_{k}\left(1-(-1)^{k}t{}^{k}z\right)^{(-1)^{k+1}b_{k}(G)}
\end{eqnarray*}
where $b_{k}\left(G\right)$ are the Betti numbers of $G$.\end{lem}
\begin{proof}
The equality \ref{eq:CheahFLS} specified at $x=1$ becomes 
\begin{eqnarray*}
\sum_{n\geq0}\sum_{\sigma\in S_{n}}\frac{z^{n}}{n!}\prod_{i=1}^{m}\det\left(I_{n}+t^{2i-1}M_{\sigma}\right)^{r_{i}} & = & \prod_{p,k}\left(1-(-1)^{k}t{}^{k}z\right)^{(-1)^{k+1}h^{k,p,p}(G)}\\
 & = & \prod_{k}\left(1-(-1)^{k}t{}^{k}z\right)^{(-1)^{k+1}\sum_{p}h^{k,p,p}(G)}\\
 & = & \prod_{k}\left(1-(-1)^{k}t{}^{k}z\right)^{(-1)^{k+1}b_{k}(G)}
\end{eqnarray*}
as wanted.
\end{proof}
The Betti number $b_{k}\left(G\right)$ equals the coefficient of
$t^{k}$ in $P_{G}\left(t\right)=\prod_{i=1}^{m}\left(1+t^{2i-1}\right)^{r_{i}}$,
so our idea is to consider those groups where this coefficient can
be interpreted in a combinatoric fashion. As mentioned, this is the
case of $G=\left(\mathbb{C}^{*}\right)^{r}$, where the Betti numbers
equal the binomial coefficient $\binom{r}{k}$. Another such example
is $G=GL\left(m,\mathbb{C}\right)$. In this case, 
\begin{eqnarray*}
P_{GL\left(m,\mathbb{C}\right)}\left(t\right) & = & \prod_{i=1}^{m}\left(1+t^{2i-1}\right)
\end{eqnarray*}
and so $b_{k}\left(GL\left(m,\mathbb{C}\right)\right)$ equals the
number of partitions of $k$ with only odd parts, no parts being repeated
or surpassing $2m-1$. Denoting this number by $p_{odd}^{n}\left(k\right)$,
we obtain:
\begin{thm}
\label{thm:combgl}Let $m\in\mathbb{N}$. Then, for formal variables
$x,z$ (or for $z,x\in\mathbb{C}$ where the series and products converge),
we have:
\begin{eqnarray*}
\sum_{n\geq0}\sum_{\sigma\in S_{n}}\frac{z^{n}}{n!}\prod_{i=1}^{m}\det\left(I_{n}-t^{2i-1}M_{\sigma}\right) & = & \prod_{k}\left(1-t{}^{k}z\right)^{(-1)^{k+1}p_{odd}^{m}\left(k\right)}
\end{eqnarray*}
where, as above, $p_{odd}^{m}\left(k\right)$ stands for the number
of partitions of $k$ with only odd parts, no parts being repeated
or surpassing $2m-1$. \end{thm}
\begin{rem}
Let $G$ be a connected linear algebraic group with Poincaré polynomial
$\prod_{i=1}^{m}\left(1+t^{2i-1}\right)^{r_{i}}$. Being the coefficient
of $\prod_{i=1}^{m}\left(1+t^{2i-1}\right)^{r_{i}}$, the Betti numbers
$b_{k}\left(G\right)$ can be interpreted in a combinatorial fahsion
for any connected linear algebraic group. Let $m\in\mathbb{N}$ and
$r^{m}=\left(r_{1},\cdots,r_{m}\right)\in\mathbb{N}_{0}^{m}$. Consider
the disjoint union $U_{m}^{r^{m}}=\bigsqcup_{i,j=1}^{m,r_{i}}\left\{ 2i-1\right\} $
and to each subset $L\subseteq U_{m}^{r^{m}}$, associate a number
$k_{L}=\sum_{j\in L}j$ ($0$ if $L=\emptyset$). Then $b_{k}\left(G\right)$
can be interpreted as the number of different subsets $L$ such that
$k_{L}=k$.\end{rem}

\end{document}